\documentclass[preprint,12pt]{article}

\usepackage{amssymb}
\usepackage{amsthm}

\usepackage{amsmath,amsfonts,amsopn,color}
\usepackage{pdfsync}
\usepackage{enumerate}

\theoremstyle{plain}
\newtheorem{theorem}{Theorem}[section]

\newtheorem{proposition}[theorem]{Proposition}

\theoremstyle{definition} 

\newtheorem{definition}[theorem]{Definition}

\theoremstyle{remark} 
\newtheorem{remark}[theorem]{Remark}

\newcommand{\new}{\newcommand}

\providecommand{\nor}[1]{\|{#1}\|}
\providecommand{\abs}[1]{\lvert{#1}\rvert}
\providecommand{\set}[1]{\{#1\}}
\providecommand{\scal}[2]{\langle{#1},{#2}\rangle}

\new{\R}{\mathbb R}
\new{\C}{\mathbb C}
\new{\N}{\mathbb N}

\new{\hh}{\mathcal H}
\new{\kk}{\mathcal K}
\new{\Ly}{\mathcal{L}(Y)}
\new{\Fy}{\mathcal{F}(X,Y)}
\new{\CbX}{\mathcal{C}(X;Y)}
\new{\la}{\lambda}
\new{\eps}{\epsilon}

\new{\beeq}[2]{\begin{equation}\label{#1}#2\end{equation}}
\new{\ran}[1]{\operatorname{Ran}#1}
\new{\Ker}[1]{\operatorname{Ker}#1}
\new{\tr}[1]{\operatorname{Tr}#1}
\new{\supp}[1]{\operatorname{supp}#1}

\new{\lspan}[1]{\operatorname{span}\{#1\}}
\new{\lspanc}[1]{\overline{\operatorname{span}}\{#1\}}
\new{\argmin}[2]{\operatornamewithlimits{argmin}\limits_{#1}{#2}}

\new{\PP}[1]{{\mathbb P}[#1]}
\new{\EE}[1]{{\mathbb E}[#1]}

\new{\X}{X}
\new{\Y}{Y}
\new{\F}{\mathcal F}
\new{\um}{1,\ldots,m}
\new{\tv}{T}
\new{\vz}{n}
\new{\rhox}{\rho_{{}_X}}
\new{\Ldue}{L^2(X,\rhox)}
\new{\f}{f_\rho}
\new{\IK}{\Phi_{\rho}}
\new{\Sx}{\Phi_n}
\new{\vy}{\mathbf y}
\new{\K}{\mathbf K}
\new{\ik}{i_\nu}
\new{\LK}{L_\nu}

\begin{document}


\title{An extension of Mercer theorem to vector-valued measurable kernels}

\author{\normalsize{Ernesto De Vito, Veronica Umanit\`a, Silvia Villa}\\
\small \em  DIMA, Universit\`{a} di Genova, Via Dodecaneso 35, 16146 Genova, Italy \\
{\small \tt  \{devito,umanita,villa\}@dima.unige.it}}


%
%


\maketitle

\begin{abstract}
  We extend the classical Mercer theorem to reproducing kernel Hilbert
  spaces whose elements are functions from a measurable space $X$into
  $\mathbb C^n$.  Given a finite measure $\mu$ on $X$, we represent
  the reproducing kernel $K$ as convergent series in terms of the
  eigenfunctions of a suitable compact operator depending on $K$ and
  $\mu$. Our result holds under the mild assumption that $K$ is
  measurable and the associated Hilbert space is
  separable. Furthermore, we show that $X$ has a natural second 
  countable topology with respect to which the eigenfunctions are continuous and
  the series representing $K$ uniformly converges to $K$ on any compact subsets
  of $X\times X$, provided that the support of $\mu$ is $X$.
\end{abstract}





\section{Introduction}
%
%

Reproducing kernel Hilbert spaces (RKHSs) are spaces of functions
defined on an arbitrary set $X$ and taking values into a normed vector
space $Y$ with the property that the evaluation operator at each point
is continuous. Usually the output space $Y$ is simply $Y=\R$ or
$\mathbb{C}$, but recently  the vector-valued setting
is becoming increasingly popular, especially in machine learning
because of its generality and its good experimental performance in a
variety of different domains  \cite{leliwa01,evmipo05,CapDev07}.  
The mathematical theory for vector-valued RKHS has been completely worked
out in the seminal paper \cite{Sch64}, which studies   the Hilbert spaces that are continuously  embedded into a locally convex topological
vector space, see also \cite{ped57}. If $Y$ is itself a Hilbert
space, the theory can be simplified as shown in
\cite{michelli05,camipoyi08,CarDevToi06,CarDevToiUma10}. In
particular, it remains true that the vector valued RKHSs are completely
characterized  by the corresponding reproducing kernel, which now
takes value in the space of bounded operators on $Y$.

The focus of this paper is on  Mercer theorem \cite{Mer09}.  In the
scalar setting, it provides a series representation,  called
\emph{Mercer representation},  for the reproducing kernel $K$ under
some suitable hypotheses. In the classical setting, $X$
is assumed to be a compact separable metric space and the reproducing
kernel $K$ to be continuous.  Hence, fixed a finite measure $\mu$  on $X$ such its
support is $X$, the integral operator $L_\mu$  with kernel $K$ is
a compact positive operator on $L^2(X,\mu)$ and  it admits an orthonormal basis $\{f_i\}_{i\in I}$  of
eigenfunctions with  non-negative eigenvalues
$\{\sigma_i\}_{i\in I}$ such that each $f_i$ with $\sigma_i>0$ is a
continuous function.  Mercer theorem states that  
\begin{equation}
K(x,t)=\sum_{i\in I} \sigma_i f_i(t)\overline{f_i(x)}\qquad \forall x,t\in X,
\label{eq:1}
\end{equation}
where the series is absolutely and uniformly convergent (see also
\cite{DunSch63}).  In the following we refer to~\eqref{eq:1} as a
Mercer representation of $K$.

The kind of representation for the reproducing kernel plays an special
r\^ole  in the applications. For
example, since the family $\{\sqrt{\sigma_i}f_i:\sigma_i>0\}$ is an
orthonormal basis of the corresponding RKHS $\hh_K$, it provides a
\emph{canonical} feature map which relates the spectral properties of
$L_\mu$ and the structure of $\hh_K$. This characterization has several consequences in the study of learning algorithms, since it allows to prove smoothing properties of kernels and to obtain error estimates, see for example \cite{cuzh07,stechr08} and references therein. In addition, the Mercer representation is an important tool in the theory of stochastic processes \cite{BerTh04,Maur08} and for dimensionality reduction methods, such as kernel PCA \cite{SchSm96,Sh05}. 

However, in many applications, the ``classical hypotheses'' of Mercer
theorem are not satisfied.  For this reason, in the recent years there
has been an increasing interest in Mercer representations under
relaxed assumptions on the input space $X$, on the kernel $K$ and on
$Y$.  A first group of results concerns scalar kernels. For example,
\cite{Sun05} dealt with the case of a $\sigma$-compact metric space
$X$ and a continuous kernel satisfying some natural integrability
conditions. When $X$ is an arbitrary measurable space endowed with a
probability measure, and $K$ is an $L^2$-integrable kernel, resorting
to the spectral properties of the operator $L_\mu$, it is possible to
obtain a Mercer representation of the kernel \cite{Wer95}. The
weakness of these results is that the corresponding series converges
only almost everywhere. More stringent assumptions on the kernel, such
as boundedness, allow to get convergence in $L^\infty$, which is still
too weak to get a pointwise representation \cite{Kon86}.  The preprint
\cite{SteSco10} contains the more general developments on the
subject. In particular, a Mercer representation enjoying pointwise
absolute convergence is obtained under less restrictive assumptions on
the kernel. More precisely, given a finite Borel measure $\mu$ on $X$
and assuming the RKHS separable and compactly embedded into
$L^2(X,\mu)$, a Mercer representation almost everywhere pointwise
convergent is recovered; moreover, it is proved that the convergence
is pointwise absolute if and only if the embedding of $\hh_K$ into
$L^2(X,\mu)$ is injective. Regarding vector valued kernels,
\cite{CarDevToi06} provides an (integral) Mercer representation under
the condition that the $K$ is square-integrable and $Y$ is a
(separable) Hilbert space.

In our paper we extend Mercer theorem in three aspects by assuming that
\begin{enumerate}[i)]
\item the input space $X$ is a measurable space;
\item the output space $Y$ is a finite dimensional vector space;
\item the kernel $K$ is a measurable function and the corresponding RKHS $\hh_K$ is separable.
\end{enumerate}
Generalizing the ideas in \cite{smzh09, DevRosToi10}, we show that $X$
has a natural second countable topology making $K$ a continuous
kernel. Moreover, fixed a finite measure $\mu$ such that its support
is $X$, we construct another measure $\nu$ such that the integral
operator $L_\nu$ of kernel $K$ is compact on $L^2(X,\nu,\C^n)$.
Hence, by using the singular value decomposition, we prove that~the
Mercer representation~\eqref{eq:1} holds true, where $\{f_i\}_{i\in
  I}$ is any orthonormal basis of eigenfunctions of $L_\nu$,
$\{\sigma_i\}_{i\in I}$ the corresponding family of eigenvalues and
the series converges uniformly on the compact subsets of $X\times X$.
If the support of $\mu$ is a proper subset of $X$,
representation~\eqref{eq:1} still holds true provided that $x,t\in
\supp{\mu}$. Note that the assumption on $Y$ can be relaxed allowing
$Y$ to be a separable Hilbert space provided that $K(x,x)$ is a
compact operator\footnote{This assumption implies that $L_\nu$ is
  compact, see Proposition 4.8 of \cite{CarDevToi06}, so that  $L_\nu$
  always has
  a basis of eigenfunctions by Hilbert-Schmidt theorem.} for all $x\in X$. However, for the sake
of clarity we state our results only for finite dimensional output
spaces and, by choosing a basis, we  can further assume that 
$Y=\C^n$.

The paper is organized as it follows. In Section $2$ we
introduce the notation and we recall some basic facts about
vector-valued reproducing kernel Hilbert spaces. Section
\ref{sec:3} contains the main results of the paper: given a measurable
vector valued reproducing kernel $K$, 
Theorem~\ref{mmercer} gives the Mercer representation of $K$
and  Proposition ~\ref{prop:vec} studies the relation between $K$ and
the scalar reproducing kernels associated with the ``diagonal blocks'' of
$K$, see~\eqref{Kj}. The proofs are given in Sections~\ref{sec:conmer} and ~\ref{sec:mmercer}.
In the former we prove the Mercer theorem for continuous vector-valued
kernels defined on metric spaces and satisfying a suitable integrability condition.
Section \ref{sec:mmercer} is devoted to the proof
of Theorem~\ref{topology} and  Proposition ~\ref{prop:vec}. The
appendix collects some properties of 
the associated integral operator.

%
%
\section{Preliminaries and notation}
%
%

For any integer $n\geq 1$, the Euclidean norm and the inner product on $\C^n$ are denoted by
  $\nor{\cdot}$ and $\langle\cdot,\cdot\rangle$. The family $\{e_j\}_{j=1}^n$  is the canonical basis of $\C^n$ and $M_n(\C)$ is the space of
  complex $n\times n$ matrices. For any matrix  $T\in M_n(\C)$ we let
  $\nor{T}=\sup\{\nor{Ty}\,:\,y\in\C^n,\,\nor{y}\leq 1\}$ be the
  operator norm, $T^*$ is the adjoint 
and $\tr{T}=\sum_{j=1}^n T_{jj}$ the trace.\\
Given a set $X$, $\mathcal{F}(X,\mathbb{C}^n)$ denotes the
  vector space of functions from $X$ into $\mathbb{C}^n$. When $X$ is endowed with a $\sigma$-algebra $\mathcal A$ and a positive finite measure $\nu:\mathcal A\to[0,+\infty)$,  then $L^2(X,\nu;\mathbb{C}^n)$ is the Hilbert space of (equivalence classes of) $\nu$-square-integrable functions from $X$ into $\C^n$,  with inner product $\scal{\cdot}{\cdot}_2$ and norm $\nor{\cdot}_2$.  If $X$ has a topology, $\mathcal{C}(X,\mathbb{C}^n)$ is the vector space of continuous functions from $X$ to $\mathbb{C}^n$ and $\mathcal{B}(X)$ is the Borel $\sigma$-algebra.

In this paper we focus on reproducing kernel Hilbert spaces whose
elements are functions from a set $X$ with values in $\C^n$. These
Hilbert spaces are completely characterized by their reproducing
kernel, which is a function from in $X\times X$ to  $M_n(\C)$,  and we
take the kernel  as the primary object.  We recall the following definition. 
\begin{definition} \label{def:k} A map $K:X\times X\to M_n(\mathbb{C})$ is called a $\mathbb{C}^n$-reproducing kernel if 
  \begin{enumerate}[a)]
  \item for all $x,t\in X$, $K(x,t)^*=K(t,x)$;
\item for any $m\geq 1$, $x_1,\ldots,x_m\in X$, $y_1,\ldots,y_m\in \mathbb{C}^n$
\[
\sum_{i,j=1}^m \langle K(x_i,x_j)y_j,y_i\rangle\geq 0.
\]
\end{enumerate}
\end{definition}
\noindent From now on we fix   a $\mathbb{C}^n$-reproducing kernel $K$ and,
for any $x\in X$ and $j=1,\ldots,n$, we denote by $K_x^j$ the function in $\mathcal{F}(X,\mathbb{C}^n)$  given by 
\[
K_x^j(t):=K(t,x)e_j, \qquad t\in X.
\]
We recall that $K$ defines a unique RKHS $\hh_K$,  whose inner product and norm of $\hh_K$ are denoted by $\langle\cdot,\cdot\rangle_K$ and $\nor{\cdot}_K$, such that $\hh_K$ is a
vector subspace of $\mathcal{F}(X,\C^n)$ and 
\begin{align}
K_x^j&\in\hh_K,&& \forall\, x\in X,\,j=1,\ldots,n \nonumber\\
\label{feature}f(x)&=\left(\scal{f}{K_x^1}_K,\ldots,\scal{f}{K_x^n}_K\right),&& \forall\, x\in X,\, f\in\hh_K,
\end{align}
see Proposition 2.1 of \cite{CarDevToi06}.  Furthermore, the following properties hold true
\begin{align}
\label{eq:ker}& K(x,t)_{l j}=\scal{K_t^j}{K_x^{l}}_K,\qquad x,t\in X \,j,l=1,\ldots,n\\
\label{eq:2}&\hh_K=\overline{\mathrm{span}}\{K_xy\,:\, x\in X, y\in \mathbb{C}^n\} \\
\nonumber& f(x)=K_x^*f \qquad x\in X
\end{align}
where $K_x:\C^n \to \hh_K$ is the (bounded) operator defined by $K_x
y=\sum y_j K^j_x$ for all $y=(y^1,\ldots,y^n)\in\C^n$.

Finally, we recall that $\hh_K$ can be realized also as
a closed subspace of some arbitrary Hilbert space by means of a
suitable feature map, as shown by the next result.
\begin{proposition}[Proposition~2.4~\cite{CarDevToi06}]\label{featuremap} 
Let $\hh$ be a Hilbert space and a map $\gamma:X\to\hh^n$. Then the
  operator $W:{\hh}\to \mathcal{F}(X;\mathbb{C}^n)$ defined by
\begin{equation}\label{defW} 
(W u)(x)=(\scal{u}{\gamma_x^1},\ldots,\scal{u}{\gamma_x^n}) , \qquad u\in\hh,\ x\in X,
\end{equation}
is a partial isometry from ${\hh}$ onto the reproducing kernel Hilbert
space $\hh_K$ with reproducing kernel
\begin{equation}
  \label{kernel_feature}
K(x,t)_{lj}=\scal{\gamma_t^j}{\gamma_x^{l}}, \qquad x,t \in X,\qquad l, j=1,\ldots,n.
\end{equation}
Moreover, $W^* W$ is the orthogonal projection onto 
\[\ker{W}^\perp=\lspanc{ \gamma_x y \mid x\in X,\ y\in \mathbb{C}^n}.\]
\end{proposition}


\section{Mercer theorem for measurable kernels}\label{sec:3}
%
%

In this section we present the main result of the paper, namely a
Mercer representation of a $\C^n$-reproducing kernel $K$ under the
assumptions that $X$ is endowed with a finite measure $\mu$ and $K$ is
measurable.  The distinctive feature of our result with respect to
already existing generalizations of Mercer theorem relies in the
construction of an {\em ad hoc} topological structure on the space
$X$, intrinsically defined by the kernel. Passing through this
topology and introducing a suitable measure related to $\mu$, we do
not assume the space $\hh_K$ to be embedded in
$L^2(X,\mu;\mathbb{C}^n)$, and we are nevertheless able to get a
Mercer representation for the kernel and a strong convergence result
on the series defining it. In particular, we recover uniform
convergence on compact subsets with respect to the topology we
introduce. 

As in \cite{smzh09,DevRosToi10}, we note that the reproducing kernel
  $K$ defines a pseudo-metric $d$ on $X$
  \begin{equation}
{d}(x,t)=\sup_{
  \begin{smallmatrix}
    y\in\C^n\\\nor{y}\leq 1
  \end{smallmatrix}
}\nor{K_xy-K_ty}_K  \qquad x,t\in X,
\label{metricad}
\end{equation}
which induces a (non-Hausdorff) topology $\tau_K$ on $X$. A basis of
$\tau_K$ is provided by the family of open balls $\set{B(x,r) \,:\,
  x\in X,r>0}$ where
\begin{equation}
B(x,r)=\{t\in X\,:\, {d}(x,t)< r\}.\label{balls}
\end{equation}
Note that the pseudo-metric $d$ can be replaced by the equivalent pseudo-metric 
$d'(x,t)=\sqrt{\sum_{j=1}^n \nor{K_x^j-K_t^j}^2_K}$, which gives rise to
  the same topology. The following result states some properties of $\tau_K$.
\begin{theorem}\label{topology}
Assume that $\hh_K$ is separable.
  \begin{enumerate}[i)]
  \item The space $X$ endowed with the topology $\tau_K$ is second countable and $K$ is continuous;
\item If $\mathcal A$ is a $\sigma$-algebra on $X$ with respect to which $K$ is measurable, the Borel $\sigma$-algebra $\mathcal{B}(X)$ generated by $\tau_K$ is contained in
 $\mathcal{A}$;
\item If $\mu:\mathcal A\to [0,+\infty)$ is a finite measure, then there exists a unique closed set $C\subset X$, namely the support of $\mu$, such that
$\mu(C)=\mu(X)$ and, if $C'$ is another closed subset with $\mu(C')=\mu(X)$, then $C'\supset C$. 
  \end{enumerate}
\end{theorem}
\noindent The support of $\mu$ is denoted by $\supp{\mu}$ and is, by its very definition, the smallest closed subset of $X$ having full measure. The assumption that $\hh_K$ is separable is essential to prove its existence. 

From now on, we fix a $\sigma$-algebra $\mathcal A$ on $X$ and a
finite measure $\mu$ defined on $\mathcal A$. We assume that $\hh_K$
is separable and $K$ is measurable, and we regard $X$ as a second
countable topological space with respect to the topology
$\tau_K$. Though $K$ is continuous,  this condition does not ensure that the integral operator with kernel $K$ is bounded on $L^2(X,\mu,\C^n)$. We overcome this problem by considering another measure $\nu$,  which is equivalent to $\mu$, such that the integral operator with kernel $K$ is bounded on $L^2(X,\nu,\C^n)$. 
Indeed, define $\nu:\mathcal A\to [0,+\infty)$ as
\begin{equation}
\nu(A):=\int_A\frac{1}{1+\nor{K(x,x)}}\mathrm{d}\mu(x),\qquad A\in\mathcal B(X)\label{nu}.
\end{equation}
Clearly $\nu$ is a positive finite measure, which is equivalent to
$\mu$ and it satisfies $\supp\nu=\supp\mu$. Furthermore, since $\tr{K(x,x)}\leq
n\nor{K(x,x)}$,  
 the integral  $\int_X\tr K(x,x)\mathrm{d}\nu(x)$ is finite and Theorem~ \ref{iK-LK} in the appendix states that the integral operator with kernel $K$
\begin{align}
\LK &: L^2(X,\nu;\mathbb{C}^n)\to L^2(X,\nu;\mathbb{C}^n) \nonumber\\
&(\LK f)(x)=\int_XK(x,t)f(t)\mathrm{d}\nu(t),\label{LKnu} 
\end{align}
 is well-defined, positive and compact\footnote{If $Y$ is infinite
   dimensional and $K(x,x)$ is compact  for all $x\in X$, it is
   possible to prove that $L_\nu$ is compact by Proposition~4.8 of
   \cite{CarDevToi06}.}.  The Hilbert-Schmidt theorem gives the
 existence of a basis of $L^2(X,\nu;\C^n)$ of eigenfunctions of $\LK$
 and this basis provides a Mercer decomposition of $K$, as shown by the
 following result.
\begin{theorem}\label{mmercer}
Let $(X,\mathcal{A})$ be a measurable space endowed with a finite
measure $\mu$.  Assume that the reproducing kernel $K:X\times X\to M_n(\mathbb{C})$ is measurable and $\hh_K$ is separable.
Define $\nu$ as in \eqref{nu} and $\LK$ as in \eqref{LKnu}. 
Then there exists a countable family $\{f_i\}_{i\in I}$ in $\mathcal F(X,\C^n)$ such that:
\begin{enumerate}[a)]
\item for all $i\in I$ the function $f_i$ is continuous with respect to $\tau_K$,
\item the family $\{f_i\}_{i\in I}$ is an orthonormal basis of $\ker{L}_\nu^\perp\subset L^2(X,\nu;\C^n)$
 and, for all $i\in I$, $\LK f_i=\sigma_if_i$ for some $\sigma_i\in (0,+\infty)$.
\end{enumerate}
Given any family $\{f_i\}_{i\in I}$ satisfying $a)$ and $b)$, then
\begin{enumerate}[i)]
\item for all  $x,t\in \supp{\mu}$ and $j,l=1,\ldots,n$
\beeq{formulamercer}
{K(x,t)_{lj}=\sum_{i\in I} \sigma_i
f^{j}_i(t)\overline{f^l_i(x)},
}
where the convergence is uniform on compact subsets of
$\supp{\mu}\times\supp{\mu}$;
\item the family $\{\sqrt{\sigma_i}f_i\}_{i\in I}$ is orthonormal in $\hh_K$;
\item if $\supp{\mu}=X$, $\{\sqrt{\sigma_i}f_i\}_{i\in I}$ is an orthonormal basis of $\hh_K$.
\item\label{parseval} for $j=1,\ldots,n$,  the family $\{\sqrt{\sigma_i}f_i^j\}_{i\in I}$ is a Parseval frame in the scalar reproducing kernel Hilbert space $\hh_{K_j}$ with reproducing kernel $K_j$ given by
\beeq{Kj}{K_j(x,t)=K(x,t)_{jj}\qquad x,t\in X.}
\end{enumerate}
\end{theorem}
\noindent We recall that $\{\sqrt{\sigma_i}f_i^j\}_{i\in I}$ is a Parseval frame in $\hh_{K_j}$ if 
\beeq{Parseval}{\nor{f}_{K_j}^2=\sum_{i\in I} \sigma_i|\langle f,f_i^j\rangle_{K_j}|^2\qquad\forall\,f\in\hh_{K_j}.}
Item \ref{parseval}) of Theorem \ref{mmercer} provides a tool to construct $\C^n$-reproducing kernels  as shown by the following result. 
\begin{proposition}\label{prop:vec}
Let $(X,\mathcal{A})$ be a measurable space endowed with a finite
measure $\mu$  such that $\supp\mu=X$.
Given a family  $K_1,\ldots,K_n$ of $n$ scalar measurable reproducing kernels on $X$, for each $j=1,\ldots,n$ take  a Parseval frame $\{f_i^j\}_{i\in I}$ in the corresponding reproducing kernel Hilbert space $\hh_{K_j}$ with $I$ countable, and define the function $K:X\times X\to M_n(\C)$ as
\beeq{defKvett}{K(x,t)_{lj}=\sum_{i\in I}
f^{j}_i(t) \overline{f^l_i(x)}\qquad \forall\,  x,t\in X.
}
The map $K$ is a measurable $\C^n$-reproducing kernel on $X$ satisfying~\eqref{Kj} and $\hh_K$ is separable.
\end{proposition}
%
\section{Continuous Mercer theorem on a metric space}\label{sec:conmer}
%
%
%
The first step in order to show Theorem \ref{mmercer} is to prove Mercer theorem under the assumption
  that $X$ is a metric space, $K$ is continuous and
  $\int_X\tr{K(x,x)}d\nu(x)$ is finite. For scalar  kernels the result is
well known, see \cite{stechr08}. However, our proof is elementary and it
holds for vector valued kernels.  As in \cite{cola06,robede10}, it is based on the singular value 
decomposition of the embedding $\ik:\hh_K\to L^2(X,\nu;\C^n)$, which
is a compact operator.
We will make use of some known properties of $\ik$ collected in the appendix.

\begin{theorem}\label{Mercer}
Let $X$ be a separable metric space and $\nu$ a finite measure defined on $\mathcal B(X)$. Assume $K:X\times X\to M_n(\mathbb{C})$ to be a continuous reproducing kernel such that 
\beeq{K-tr-cl}{\int_X\tr K(x,x)\,\mathrm{d}\nu(x)<+\infty.}
Define the trace class operator $\LK$ as in~\eqref{LKnu} and 
take  an orthonormal  basis $\{f_i\}_{i\in I}$  of
$\,\ker{\LK}^\perp$ of  continuous eigenvectors of $\LK$ and let
$\{\sigma_i\}_{i\in I}\subseteq (0,+\infty)$ be the corresponding family of eigenvalues. Then the family $\{\sqrt{\sigma_i}f_i\}_{i\in I}$ is orthonormal  in $\hh_K$ and
\beeq{eq:mercer}
{K(x,t)_{lj}=\sum_{i\in I} \sigma_i
f^{j}_i(t) \overline{f^l_i(x)}\qquad \forall\, x,t\in \supp{\nu},}
where the series converges uniformly on any compact subset of
$\supp{\nu}\times\supp{\nu}$.
If $\supp{\nu}=X$, $\{\sqrt{\sigma_i}f_i\}_{i\in I}$ is an orthonormal basis of $\hh_K$.
\end{theorem}
\begin{remark}
Item 4) of Theorem \ref{iK-LK} in the appendix guarantees the existence of a basis $\{f_i\}_{i\in I}$ of $\,\ker{\LK}^\perp$
of continuous eigenvectors of $\LK$.
\end{remark}
\begin{proof} As in Theorem \ref{iK-LK},  we denote by  $\ik:\hh_K\hookrightarrow L^2(X,\nu;\mathbb{C}^n)$ the canonical embedding. Its adjoint $\ik^*$ is given by \eqref{i*}, so that
$\LK=\ik\ik^*$, and we define the operator $T_\nu:\hh_K\to\hh_K$ as $T_\nu:=\ik^*\ik$. 
Take a family $\{f_i\}_{i\in I}$ as in the statement of the theorem and, for all $i\in I$, define $g_i=i^*_K f_i/\sqrt{\sigma_i}$.  The singular value decomposition of $\ik^*$ gives that $\{g_i\}_{i\in I}$  is an orthonormal basis  of $\,\ker{T_\nu}^\perp$ of  eigenvectors of $T_\nu$.  We claim that, for all $x\in\supp{\nu}$ and $j=1,\ldots,n$, $K_x^j\in \,\ker{T_\nu}^\perp$. Indeed,  for any $f\in \ker{T_\nu}$
\begin{align*}
0= \langle T_\nu f,f\rangle_K& = \scal{\ik f}{\ik f} =\sum_{j=1}^n \int_X \abs{f^j(x)}^2 \,\mathrm{d}\nu(x).
\end{align*}
Hence, for any $j=1,\ldots,n$,  the map $x\mapsto f^j(x)=\langle f,K_x^j \rangle_K$ is zero $\nu$-almost everywhere.  Since $\hh_K\subseteq C(X,\C^n)$, see item~1) of Theorem \ref{iK-LK}, the definition of support implies that
$\langle f,K_x^j \rangle_K=0$ for all $x\in\supp{\nu}$. Hence
  \begin{equation}
    \label{nucleo}
    K_t^j\in\mathrm{ker}T_\nu^\perp\qquad \forall t\in \supp{\nu},\quad j=1,\ldots,n.
  \end{equation}
Furthermore, since $\set{g_i}_{i\in I}$ is a basis of $\mathrm{ker}T_\nu^\perp$, for all $x\in\supp{\nu}$  and $j=1,\ldots,n$
\[
K^j_x=\sum_{i\in I}\langle K^j_x, g_i\rangle_K \, g_i.
\] 
Hence, the reproducing property gives that 
$$
K(x,t)_{lj}=\langle K^l_x,K^j_t \rangle_K=\sum_{i\in I}\langle K^l_x, g_i\rangle_K\langle g_i,K^j_t\rangle_K=\sum_{i\in I}\sigma_i f_i^j(t) \overline{f_i^l(x)}
$$
for all $x,t\in\supp{\nu}$.

Concerning the uniform convergence, suppose $I=\mathbb{N}$, fix two compact subsets $C,C'\subseteq \mathrm{supp}\nu$, and consider the remainder
\beeq{eq:mag}{
\sup_{(x,t)\in C\times C'}\left|\sum_{i=q}^{+\infty} \sigma_if_i^j(t)\overline{f_i^l(x)} \right|\leq  \sqrt{\sup_{x\in C}\sum_{i=q}^{+\infty} \sigma_i\abs{f_i^l(x)}^2}\sqrt{\sup_{t\in C'}\sum_{i=q}^{+\infty} \sigma_i\abs{f_i^j(t)}^2}.
}
The series of continuous functions $\sum_{i=0}^{+\infty} \sigma_i\abs{f_i^l(x)}^2$ converges pointwise to the continuous function $K(x,x)_{ll}$ on the compact set $C$, and therefore
uniform  convergence follows from Dini's theorem. Thus, relying on the bound in \eqref{eq:mag}, we have
\[
\lim_{q\to +\infty}\sup_{(x,t)\in C\times C'}\left|\sum_{i=q}^{+\infty} \sigma_if_i^j(t) \overline{f_i^l(x)}\right|=0.
\] 
Assume that $\supp{\nu}=X$.  Since, by~\eqref{eq:2}, 
$\set{K_t^j:t\in X,j=1,\ldots n}$ is total in $\hh_K$,~\eqref{nucleo} implies that
$\mathrm{ker} T_\nu=\{0\}$. Hence  the family $\{\sqrt{\sigma_i}f_i\}_{i\in I}$ is an orthonormal basis of $\hh_K$.
\end{proof}

\section{Proofs}\label{sec:mmercer}
%
%
%
To prove the Mercer representation in the general setting of
  Theorem~\ref{mmercer},  we would like to define a metric $d$ on $X$
 such that $K$ becomes continuous. A natural choice would be the map $d$ defines by~\eqref{metricad}.
However, $d$ is not a metric unless the map $x\mapsto K_x$ is
injective.   To overcome this problem, we first
introduce a suitable metric space  $\widetilde{X}$ and a 
  continuous kernel $\widetilde{K}$  such that the corresponding
  reproducing kernel $\hh_{\widetilde{K}}$ is  isomorphic to $\hh_K$ and, as a
  consequence, we prove  Theorem~\ref{topology}. Afterwards, 
the Mercer
  representation of $K$ is deduced by the corresponding representation~\eqref{eq:mercer}
  of $\widetilde{K}$ given by Theorem~\ref{Mercer}.  From now on $(X,\mathcal{A})$ is a measurable space endowed  with a finite
measure $\mu$ and $K$ is a $\C^n$-measurable reproducing kernel such that $\hh_K$ is separable.

Clearly ${d}$ in \eqref{metricad} is a pseudo-metric. 
The symmetry property and the triangular inequality directly follow from the definition, while from $d(x,t)=0$ we get $K_x=K_t$, which as noted before in general does not imply $x=t$. However, the reproducing property  \eqref{feature} gives $f(x)=f(t)$ for all $f\in\hh_K$, which means that the functions in $\hh_K$ are not able to distinguish the points $x$ and $t$.
This suggests to define an equivalence relation $\sim$  on $X$ by setting 
\begin{equation}
x\sim t  \iff  K_x=K_t.
\end{equation}
Denote by  $\widetilde{X}=X{/_\sim}$ the corresponding quotient space and,  given $[x], [t] \in \widetilde{X}$,  define the function $\tilde{d}([x],[t]):={d}(x,t)$. Then $\tilde{d}$ is a distance on $\widetilde{X}$ so that $(\widetilde{X},\tilde{d})$ is a metric space.

We consider the pull-back topology $\tau_K$ induced on $X$ by the canonical projection $\pi:X\to X/_\sim$, i.e.
$$\tau_K=\{\pi^{-1}(A)\,:\,A\,\text{open in }(\widetilde{X},\tilde{d})\}.$$
 It is clear that   the family of open balls $\set{B(x,r) \,:\,
  x\in X,r>0}$ is a basis for $\tau_K$, see~\eqref{balls}. Now, the proof
of Theorem~\ref{topology} is  a consequence of the next proposition
where $\mathcal{L}(\mathbb{C}^n,\hh_K)$ denotes the space of (bounded)
linear operator from $\C^n$ to $\hh_K$ endowed with the operator 
norm $\nor{\cdot}_{{}_{n,K}}$ so that, for example, 
\[ d(x,t)=\nor{K_x-K_t}_{{}_{n,K}}=  \sup_{\begin{smallmatrix}
    y\in\C^n\\\nor{y}\leq 1
  \end{smallmatrix}
}\nor{K_xy-K_ty}_K. \]
\begin{proposition}\label{tauK}The following facts hold:
\begin{itemize}
\item[i)] the map $\Phi: \widetilde{X}\to
  \mathcal{L}(\mathbb{C}^n,\hh_K)$ given by $\Phi([x])=K_x$ is an
  isometry from $(\widetilde{X},\widetilde{d})$ into
  $(\mathcal{L}(\mathbb{C}^n,\hh_K), \nor{\cdot}_{{}_{n,K}})$;
\item[ii)] the spaces $(\widetilde{X},\tilde{d})$ and $(X,\tau_K)$ are second countable;
\item[iii)] the $\sigma$-algebra $\mathcal{A}$ contains $\mathcal{B}(X)$, the Borel sets generated by $\tau_K$;
\item[iv)] given a positive finite measure $\nu$ on $X$, there exists $\supp\nu$.
\end{itemize}
\end{proposition}
\begin{proof} Statement {\em i)}  follows directly from the definition of the equivalence relation $\sim$ and the pseudo-distance $d$. \\
$ii)$ Since $\hh_K$ is separable, the space
$\mathcal{L}(\mathbb{C}^n,\hh_K)$ can be identified with $\hh_K^n$, and
then it is separable. Therefore, the set $\Phi(\widetilde{X})\subseteq
\mathcal{L}(\mathbb{C}^n,\hh_K)$ is separable as well, and so is
$\widetilde{X}$, since $\Phi$ is an isometry. Since $\widetilde{X}$ is a
separable metric space, there exists a countable basis $\{A_i\}_{i\in
  N}$ of open subsets of $\widetilde{X}$.  Clearly,  $\{\pi^{-1}(A_i)\}_{i\in
  N}$ is a countable basis for $\tau_K$, that is,  $\tau_K$ is second
countable.

To show that {\em iii)} holds true, it is enough to prove that each
element $B(x,r)$ of the basis of $\tau_K$ belongs to
$\mathcal{A}$. Towards this end, if for a given $x\in X$ we prove that
the map $G_x:(X,\mathcal{A})\to [0,+\infty)$,
$G_x(y)=\nor{K_y-K_x}_{{}_{n,K}}$ is measurable we are done. Since
$G_x$ is the composition of the function $X\ni y\mapsto K_y-K_x\in
\mathcal{L}(\mathbb{C}^n,\hh_K)$, with
$\mathcal{L}(\mathbb{C}^n,\hh_K)\ni A \mapsto \nor{A}_{{}_{n,K}}\in
\mathbb{R}$, and  the latter is continuous, it is enough to prove that
the first one is measurable. This follows from separability of $\hh_K$
and Proposition 3.1 in \cite{CarDevToi06}.

Finally, to prove $iv)$, define $\supp{\nu}$ as the intersection of all $\tau_K$-closed subsets $C\subseteq X$ with $\nu(C)=\nu(X)$. Clearly $\supp{\nu}$ is closed, and we prove that $\nu(\supp{\nu})=\nu(X)$.
Indeed, since $\tau_K$ is second countable, there exists a sequence of closed sets $\{C_j\}_{j\in\mathbb{N}}$ such that, for an arbitrary closed set $C$, $C=\cap_k C_{j_k}$ for a suitable subsequence  $\{C_{j_k}\}_{k\in N}$. Hence,

\begin{align*}
  \nu(\supp{\nu}) &=\nu\Bigg(\bigcap_{{\scriptsize\begin{tabular}{c}\text{$C$ closed,}\\ $\nu(C)=\nu(X)$\end{tabular}}} C\,\Bigg) =\nu\Bigg(\bigcap_{{\scriptsize\begin{tabular}{c}$j\!\in\!\mathbb{N}$\\ $\nu(C_j)=\nu(X)$\end{tabular}}}\!\!\!C_j\Bigg)\\
\\ 
& =\lim_{{{\scriptsize\begin{tabular}{c}$j\!\in\!\mathbb{N}$\\ $\nu(C_j)=\nu(X)$\end{tabular}}}}\hspace{-0.4cm}\nu(C_j)\hspace{0.1cm}=\nu(X).
\end{align*}
\end{proof}
Note that, since $\widetilde{X}$ is a second countable metric space, it
is separable.  We now define a continuous kernel $\widetilde{K}$ on the separable metric space $(\widetilde{X},\widetilde{d})$ in order to apply Theorem~\ref{Mercer}, once that a suitable measure $\tilde{\nu}$ has been also introduced. Set
\[
\widetilde{K}:\widetilde{X}\times \widetilde{X}\to \Ly, \qquad \widetilde{K}([x],[t]):=K(x,t),
\] 
and denote by $\hh_{\widetilde{K}}$ the RKHS associated to
$\widetilde{K}$. First of all, note that~\eqref{eq:ker} and
the definition of the equivalence classes in $\widetilde{X}$ guarantee that
$\widetilde{K}$ is well-defined. The next proposition aims at
clarifying some basic properties of this space and most of all the
connections between $\hh_K$ and $\hh_{\widetilde{K}}$. In particular,
as it will be made precise later, the two spaces roughly speaking
coincide. 

\begin{proposition}\label{Ktilde}
The following facts hold:
\begin{itemize}
\item[i)] ${\widetilde{K}}$ is a continuous kernel and every $f\in\hh_{\widetilde{K}}$ is a continuous function;
\item[ii)] $\hh_{\widetilde{K}}$ is separable;
\item[iii)] $\hh_{\widetilde{K}}$ and $\hh_K$ are unitarily equivalent by means of the unitary operator
\beeq{W}{W:\hh_{\tilde{K}}\to \hh_K\qquad (W\tilde{f})(x):=\tilde{f}([x]);}
\item[iv)] given a sequence of functions $(\widetilde{f}_n)_{n\in
      \mathbb{N}}$ in $\hh_{\widetilde{K}}$ such that $\widetilde{f}_n\to \widetilde{f}\in\hh_{\widetilde{K}}$
  uniformly on the compact sets of $\widetilde{X}$, then $W\widetilde{f}_n\to W\widetilde{f}$
  uniformly on the compact sets of $X$.  
\end{itemize}
\end{proposition}
\begin{proof}
{\em i)} Given $x_0,t_0\in X$ we prove that $\widetilde{K}$ is continuous in $([x_0],[t_0])$. For all $x,t\in X$ we have
\begin{eqnarray*}\nor{\widetilde{K}([x],[t])-\widetilde{K}([x_0],[t_0])}&\leq&
\nor{K_x^*K_t-K_x^*K_{t_0}}+\nor{K_x^*K_{t_0}-K_{x_0}^*K_{t_0}}\\
&\leq&\!\!\!\nor{K_x^*}_{{}_{K,n}}\nor{K_t-K_{t_0}}_{{}_{n,K}}+\!\!\nor{K_x^*-K_{x_0}^*}_{{}_{K,n}}\nor{K_{t_0}}_{{}_{n,K}}.
\end{eqnarray*}
Since $\nor{K_x^*}_{{}_{K,n}}\leq \nor{K_x^*-K_{x_0}^*}_{{}_{K,n}}+\nor{K_{x_0}^*}_{{}_{K,n}}=\nor{K_x-K_{x_0}}_{{}_{n,K}}+\nor{K_{x_0}}_{{}_{n,K}}$, the continuity of $\Phi$ gives the thesis.\\ The second part of statement $i)$ follows by the reproducing formula $f(x)=\widetilde{K}_x^*f$ for all $f\in\hh_{\widetilde{K}}$.\smallskip

$ii)$ Since $\widetilde{X}$ and $\mathbb{C}^n$ are separable (see Proposition \ref{tauK}.$ii$), the space $\hh_{\widetilde{K}}=\overline{\mathrm{span}}\{\widetilde{K}_{[x]}y\,:\, x\in X, y\in \mathbb{C}^n\}$ is separable too.\smallskip

$iii)$ We apply Proposition \ref{featuremap} taking $\hh=\hh_{\widetilde{K}}$ and $\gamma_x=\widetilde{K}_{[x]}$, so that 
$$(W\tilde{f})(x)=\tilde{f}([x])=f(x)$$ for all $x\in X$ and $\tilde{f}\in \hh_{\widetilde{K}}$. Since $$\gamma_x^*\gamma_t=\widetilde{K}([x],[t])=K(x,t),$$ the operator $W$ is a partial isometry from $\hh_{\widetilde{K}}$ into $\hh_K$; moreover, $f=0$ clearly implies $\tilde{f}=0$, and so $W$ is injective.\smallskip

$iv)$  Let $C$ be a compact subset of $(X,\tau_K)$. Since by construction $\pi:X\to \widetilde{X}$ is continuous with respect to $\tau_K$, $\pi(C)$ is compact in $\widetilde{X}$ and therefore $\sup_{[x]\in \pi(C)} |\widetilde{f}_n([x])- \widetilde{f}([x])|\to 0$. Being  by definition $W\widetilde{f}_n(x)=\widetilde{f}_n([x])$, the thesis follows. 
\end{proof}
 
In order to apply Theorem \ref{Mercer} to the kernel $\widetilde{K}$,
the last ingredient we need is a finite measure $\tilde{\nu}$ on $\widetilde{X}$. If $\nu$ is defined as in \eqref{nu},  using the canonical projection we can set
$$\tilde{\nu}(A):=\nu(\pi^{-1}(A))\qquad\text{for all Borel set $A$ in $(\widetilde{X},\widetilde{d})$}.$$
$\tilde{\nu}$ is well defined since $\pi^{-1}(A)\in \mathcal{B}(X)$ being $\pi$ continuous,  and $\mathcal{B}(X)\subseteq\mathcal{A}$ thanks to Proposition \ref{tauK}.$iii)$. Moreover, we clearly have
\beeq{supp}{\supp\mu=\supp{\nu}=\pi^{-1}(\supp{\tilde{\nu}}).}
We are now ready to prove our main result.

\begin{proof}[Proof of Theorem~\ref{mmercer}]
 From the results
  collected  so far, we know that $\widetilde{K}$
  is a continuous kernel by Proposition \ref{Ktilde}, and
  $(\widetilde{X},\widetilde{d})$ is a separable metric space (see
  Proposition \ref{tauK}), which is endowed with a finite measure
  $\tilde{\nu}$.  In order to apply Theorem \ref{Mercer}, we need to
  show that the integrability condition \eqref{K-tr-cl} is met by
  $\widetilde{K}$.  From the definition of $\widetilde{X}$,
  $\widetilde{K}$ and $\tilde{\nu}$, taking into account that
  $\widetilde{K}([x],[x])=K(x,x)$ for all $x\in X$, and using the
  change of variables $[x]=\pi(x)$, we have
  \[
  \int_{\widetilde{X}}
  \widetilde{K}([x],[x])\,\mathrm{d}\tilde{\nu}([x])=\int_{X} K(x,x)
  \mathrm{d}{\nu}(x).
  \]
  Therefore $\widetilde{X}$, $\widetilde{K}$ and $\tilde{\nu}$ satisfy
  the assumptions of Theorem \ref{Mercer}. Hence, $\widetilde{K}$ can be written component-wise as
  \beeq{eq:sek}{\widetilde{K}([x],[t])_{jl}=\sum_{i\in
      I}\sigma_i\tilde{f}^l_i([t])\overline{\tilde{f}^j_i([x])}} where
  $(\sqrt{\sigma}_i\tilde{f}_i)_{i\in I}$ is basis of
    $\ker{L_{\tilde{\nu}}}^\perp$  of eigenvectors of the
  integral operator $L_{\tilde{\nu}}$ whose kernel is $\widetilde{K} $.
  Furthermore $(\sqrt{\sigma}_i\tilde{f}_i)_{i\in I}$ is an orthonormal
  family of $\hh_{\widetilde{K}}$ with $\tilde{f}_i$ continuous on
  $\widetilde{X}$.  Then $f_i:=W\tilde{f}_i$ is a continuous function
  on $X$ thanks to the definition of $\tau_K$ and $W$ (see \eqref{W}),
  and $(\sqrt{\sigma}_if_i)_{i\in I}$ is an orthonormal part of
  $\hh_{K}$ by Proposition \ref{Ktilde}.$iii)$. Moreover, for all
  $f,g\in\hh_K$, it holds $$\int_{\widetilde{X}}
  \tilde{f}([x])\tilde{g}([x]) \mathrm{d}\tilde{\nu}([x])=\int_X
  f(x)g(x)\mathrm{d}\nu(x)$$ by definition of $\tilde{\nu}$, and
  thus $\{f_i\}_{i\in I}$ is an orthonormal family in
  $L^2(X,\nu;\C^n)$ as well. Note that, $\{f_i\}$ is also a basis of eigenvectors of
  $L_\nu$ since $(L_\nu f_i)(x)=(L_{\tilde{\nu}}\tilde{f_i})([x])$ for
  all $i\in I$. The definition of $W$ and equation \eqref{supp}
  entail $$K(x,t)_{jl}=\sum_{i\in I}\sigma_if^l_i(t) \overline{f^j_i(x)}$$
  for all $x,t\in\supp\nu=\supp\mu$.\\
  Since the series in \eqref{eq:sek} is uniformly convergent on the
  compact subsets of $\supp{\tilde{\nu}}\times\supp{\tilde{\nu}}$, by
  Proposition \ref{Ktilde}.$iv$ the latter series is uniformly
  convergent on the compact subsets of $\supp{\mu}\times\supp{\mu}$.

 The unitary equivalence between $\hh_K$ and
  $\hh_{\widetilde{K}}$ (through $W$) implies that
  $(\sqrt{\sigma}_if_i)_{i\in I}$ is an orthonormal basis of $\hh_K$
  if and only if $(\sqrt{\sigma}_i\tilde{f}_i)_{i\in I}$ is an
  orthonormal basis of $\hh_{\widetilde{K}}$.  Hence item~$iii)$
  is a consequence of Theorem~\ref{Mercer} and~\eqref{supp}.

  Finally, we prove item~$\ref{parseval}$). First of all note that it
  straightforward to see that every $K_j$ given by \eqref{Kj} is a
  scalar kernel on $X$. Moreover, it satisfies
$$K_j(x,t)=\sum_{i\in I}
\sigma_if^{j}_i(t) \overline{f^j_i(x)}\qquad \forall\, x,t\in X$$
thanks to equation \eqref{formulamercer}.\\
Fix $j=1,\ldots,n$ and set $\gamma_x=(\sqrt{\sigma_i}f_i^j)_{i\in
  I}\in\ell^2(I)$ for all $x\in X$. Since
$$\scal{\gamma_t}{\gamma_x}=\sum_{i\in I}\sigma_if^{j}_i(t) \overline{f^j_i(x)}=K_j(x,t),$$
the function defined by $$(W^jc)(x):=\scal{c}{\gamma_x}=\sum_{i\in
  I}\sqrt{\sigma_i}c_i\overline{f^j_i(x)},\qquad c\in\ell^2(I),$$ is a
partial isometry onto $\hh_{K_j}$ by Proposition
\ref{featuremap}. Therefore we have
$$\nor{f}^2_j=\nor{W^*f}^2_2=\sum_{i\in I} \sigma_i|\langle f,f_i^j\rangle_{K_j}|^2\qquad\forall\,f\in\hh_{K_j},$$
i.e. $\{\sqrt{\sigma_i}f_i^j\}_{i\in I}$ is a Parseval frame in
$\hh_{K_j}$.
\end{proof}

\begin{proof}[Proof of Proposition~$\ref{prop:vec}$]
  Fix $j=1,\ldots,n$ and let $\{f_i^j\}_{i\in I}$ be a Parseval frame in $\hh_{K_j}$. The function $K$ given by \eqref{defKvett} is a $\C^n$-reproducing kernel on $X$ since
\begin{eqnarray*}\sum_{l,r=1}^m\scal{K(x_l,x_r)y_r}{y_l}&=&\sum_{l,r=1}^m\sum_{p,q=1}^nK(x_l,x_r)_{pq}y_r^q\overline{y_l^{p}}\\
&=&\sum_{l,r=1}^m\sum_{p,q=1}^n\sum_{i\in I}y_r^q\overline{y_l^{p}}f^{q}_i(x_r) \overline{f^p_i(x_l)}\\&=&\sum_{i\in I}\left |\sum_{r=1}^m\sum_{q=1}^ny_r^qf^{q}_i(x_r)\right |^2\geq 0
\end{eqnarray*}
for all $x_1,\ldots,x_m\in X$, $y_1,\ldots,y_m\in\C^n$, $m\geq 1$.
Finally, we have
$$K_j(x,x)=\nor{(K_j)_x}_j^2=\sum_{i\in I} |f_i^j(x)|^2=K(x,x)_{jj}$$
for all $x\in X$, so that $K_j(x,t)=K(x,t)_{jj}$ for all $x,t\in X$ by
polarization's identity.

Since all $K_j$ are measurable, so is $K$. The fact that $I$ is
countable implies that each $\hh_{K_j}$ are separable as well as $\hh_K$.
\end{proof}
%
%
\section{Appendix}
%
%
We recall some basic facts about the embedding of a reproducing kernel
Hilbert space into $L^2(X,\nu,\C^n)$.
\begin{theorem}\label{iK-LK}
Let $X$ be a separable metric space and $\nu$ a finite measure on
$X$. Assume $K:X\times X\to M_n(\mathbb{C})$ to be a
$\C^n$-reproducing kernel such that it is continuous and 
\beeq{K-tr-cl2}{\int_X\tr K(x,x)\,\mathrm{d}\nu(x)<+\infty.} 
The following facts hold true:
\begin{enumerate}
\item every function in $\hh_K$ is continuous and $\hh_K$ is separable;
\item the canonical embedding 
\beeq{eq:ik}{
\ik:\hh_K\hookrightarrow L^2(X,\nu;\mathbb{C}^n)
}
is a well defined compact operator. Its adjoint $\ik^*:L^2(X,\nu;\mathbb{C}^n)\to \hh_K$ is given by 
\beeq{i*}{
\ik^* f=\sum_{j=1}^n \int_X K_x^j f^i(x)\,\mathrm{d}\nu(x),
}
where the integrals converge in $\hh_K$;
\item
the composition $\ik\ik^*:L^2(X,\nu;\mathbb{C}^n)\to L^2(X,\nu;\mathbb{C}^n)$ is a positive trace class operator given by 
\[
(\ik\ik^*f)(x)=\int_X K(x,t)f(t)\mathrm{d}\nu(t)=(\LK f)(x);
\]  
\item there exist a family $\{f_i\}_{i\in I}$ of $L_K$ in $\,\mathcal{C}(X,\mathbb{C}^n)\cap L^2(X,\nu;\C^n)$ and a sequence  $\{\sigma_i\}_{i\in I}$ in $(0,+\infty)$ such that $\{f_i\}_{i\in I}$ is an orthonormal basis of $\ker{\LK}^\perp=\overline{\mathrm{Ran}\LK}$ and
\[
\LK f_i=\sigma_i f_i\qquad\forall\,i\in I.
\]
\end{enumerate}
\end{theorem}

\begin{proof} Set $M:=\int_X\tr K(x,x)\,\mathrm{d}\nu(x)\in\R_+$.
\\
\noindent$1.$\,  Given $f\in\hh_K$, by the reproducing property 
\[
f(x)=(\langle f,K_x^1\rangle_K,\ldots,\langle f,K_x^n\rangle_K).
\]
Since the $j$-th component of $f$ coincides  with the composition of
the inner product in $\hh_K$ with the map $x\mapsto K_x^j$, which is
clearly continuous,  it follows that $\hh_K\subseteq
\mathcal{C}(X,\mathbb{C}^n)$. Moreover, since $X$ is separable, there
exists a countable set dense $X_0$ dense in $X$. Hence $\hh_K$
is separable since $\mathcal S=\{K_x^j:x\in X_0,j=1,\ldots,n\}$ is total in
$\hh_K$. Indeed, take $f\in \mathcal S^\perp$, then the reproducing
property gives  that $f(x)^j=\scal{f}{K_x^j}=0$ for all $x\in X_0$ and
$j=1,\ldots,n$. Since $f$ is continuous and $X_0$ dense, it follows
that $f=0$, so that the claim is proved.\\
{$2.$\,} If $f\in\hh_K$, then the following chain of inequalities holds:
\beeq{limitat}{\!\!\int\limits_X \nor{f(x)}^2\,d\nu(x) \leq\!\! \int\limits_X \scal{K_xK_x^*f}{f}^2_{K} \,d\nu(x) \leq\!\! \int\limits_X \nor{f}^2_{K} \tr{K(x,x)} \,d\nu(x) \leq\!\! M \nor{f}^2_{K},}
and the last quantity is finite by hypothesis. Thus $i_K$ is well-defined and bounded. 
Moreover, if $f\in L^2(X,\nu;\mathbb{C}^n)$, we get
$$\scal{\ik^*f}{g}_K=\scal{f}{g}_2=\int_X\scal{f(x)}{K_x^*g}\mathrm{d}\nu(x)=\scal{\int_XK_xf(x)\mathrm{d}\nu(x)}{g}_K,$$
where the integral $\int_XK_xf(x)\mathrm{d}\nu(x)$ converges in $\hh_K$ by the H\"older inequality, since
$$\int_X\nor{K_xf(x)}\mathrm{d}\nu(x)\leq  \int_X \big(\tr K(x,x)\big)^{1/2} \nor{f(x)}\,d\nu(x),$$
 $x\mapsto \big(\tr K(x,x)\big)^{1/2}\in L^2(X,\nu;\mathbb{C}^n)$ and  $f\in L^2(X,\nu;\mathbb{C}^n)$. The component-wise representation in equation \eqref{i*} follows by \eqref{feature}.
 
\ \\
\noindent $3.$\, The formula for $\LK$ follows immediately using the expression for $\ik^*$ obtained in item 1 and the  fact that $\ik$ is the canonical embedding.  In order to prove that $\LK$ is a Hilbert-Schmidt operator, we prove that in fact is a trace class operator.   Fix $\{\varphi_\ell\}_{\ell\in\mathbb{N}}$ an orthonormal basis of $\hh_K$ and note that
\begin{align*}
%
%
\tr \LK=\tr (\ik^*\ik)&=\sum_{\ell\in\mathbb{N}} \nor{i_K \varphi_\ell}_2^2 =\sum_{\ell\in{\mathbb{N}}} \int_X \nor{\varphi_\ell(x)}^2 \,d\nu(x) \\
&=\sum_{\ell\in\mathbb{N}}\int_X \sum_{j=1}^n \langle \varphi_\ell,K_x^j \rangle_K^2 \,d\nu(x)\\
&=\int_X \sum_{j=1}^n\sum_{\ell\in\mathbb{N}} \langle \varphi_\ell,K_x^j \rangle_K^2 \,d\nu(x)\\
&=\int_X \sum_{j=1}^n\nor{K_x^j }_K^2 \,d\nu(x)\\
&=\int_X \sum_{j=1}^n K(x,x)_{jj} \,d\nu(x)\\
&= M.
\end{align*} 

Then $\LK$ is compact and being positive by construction, there exist a  basis of eigenvectors $\{f_i\}_{i\in \mathbb{N}}\subseteq L^2(X,\nu;\mathbb{C}^n)$ and the associated sequence of positive eigenvalues $\{\sigma_i\}_{i\in \mathbb{N}}$.  If we denote by $I$ the set of indices corresponding to strictly positive eigenvalues, we have that $\{f_i\}_{i\in I}$ is a basis of $\ker{\LK}^\perp=\overline{\mathrm{Ran}\LK}$. 
On the other hand, given $f_i$, $i\in I$, if we define 
$g_i=i^*_K f_i/\sqrt{\sigma_i}\in\hh_K$,
we have that $\ik g_i\in C(X,\mathbb{C}^n)$ (i.e. it admits a continuous representative) by $1$, and \beeq{g-i-f-i}{(\ik g_i)(x)=(\LK f_i)(x)/\sqrt{\sigma_i}=\sqrt{\sigma_i}f_i(x)}
for $\nu$-almost all $x\in X$.
Thus, we can assume without loss of generality $f_i$ to be continuous for all $i\in I$.  

\end{proof}


\end{document}